\documentclass[a4paper,11pt]{amsart}

\usepackage{amsfonts,amssymb,mathtools,amsmath}
\usepackage{textcomp, color}

\theoremstyle{plain}

\newtheorem*{thmA}{Theorem A}

\newtheorem{thm}{Theorem}[section]
\newtheorem{lem}[thm]{Lemma}

\theoremstyle{definition}

\newtheorem{dfn}[thm]{Definition}

\newcommand{\Z}{\mathbb{Z}}

\newcommand{\N}{\mathbb{N}}

\DeclareMathOperator{\Aut}{Aut}

\DeclareMathOperator{\Cl}{Cl}

\begin{document}

\title{An infinite family of strongly real Beauville $p$-groups}

\author[\c{S}.\ G\"ul]{\c{S}\"ukran G\"ul}
\address{Department of Mathematics\\ Middle East Technical University\\
	06800 Ankara, Turkey}
\email{gsukran@metu.edu.tr}

\keywords{Strongly real Beauville groups; triangle group\vspace{3pt}}

\thanks{The author is supported by the Spanish Government, grant MTM2014-53810-C2-2-P, the Basque Government, grant IT974-16, and the ERC Grant PCG-336983.}

\begin{abstract}
	We give an infinite family of non-abelian strongly real \break Beauville $p$-groups for every prime $p$ by considering the quotients of triangle groups, and indeed we prove that there are non-abelian strongly real Beauville $p$-groups of order $p^n$ for every $n\geq 3,5$ or $7$ according as $p\geq 5$ or $p=3$ or $p=2$.  This shows that there are strongly real Beauville $p$-groups exactly for the same orders for which there exist Beauville $p$-groups.
\end{abstract}	
	
\maketitle
	
\section{Introduction}

Let $G$ be a finite group. For a couple of elements $x,y \in G$, we define
\[
\Sigma(x,y)
=
\bigcup_{g\in G} \,
\Big( \langle x \rangle^g \cup \langle y \rangle^g \cup \langle xy \rangle^g \Big),
\]
that is, the union of all subgroups of $G$ which are conjugate to $\langle x \rangle$, to 
$\langle y \rangle$ or to $\langle xy \rangle$. Then $G$ is called a \emph{Beauville group\/} if the following conditions hold:
\begin{enumerate}
	\item $G$ is a $2$-generator group.
	\item There exists a pair of generating sets $\{x_1,y_1\}$ and $\{x_2,y_2\}$ of $G$ such that 
	$\Sigma(x_1,y_1) \cap \Sigma(x_2,y_2)=1$.
\end{enumerate}
Then $ \{x_1,y_1\}$ and $\{x_2,y_2\}$ are said to form a \emph{Beauville structure\/} for $G$.  We call $\{x_i, y_i, x_iy_i\}$ the \textit{triple} associated to $\{x_i, y_i\}$ for $i=1,2$.
The \textit{signature\/} of a triple is the tuple of orders of the elements in the triple.

Every Beauville group gives rise to a complex surface of general type which is known as a \emph{Beauville surface\/} of unmixed type. If we want Beauville surfaces to be real, then we work with the following sufficient condition on the Beauville structure.

\begin{dfn}
	Let $G$ be a Beauville group. We say that $G$ is \emph{strongly real\/} if there exist a Beauville structure $\{\{x_1,y_1\}, \{x_2,y_2\} \}$, an automorphism $\theta \in \Aut(G)$ and elements $g_i \in G$ for $i=1,2$ such that 
	\[
	g_i \theta(x_i)g_i^{-1}=x_i^{-1} \ \ \text{and} \ \ g_i \theta(y_i)g_i^{-1}=y_i^{-1}
	\]
	for $i=1,2$. Then the Beauville structure is called \emph{strongly real}.
\end{dfn}
In practice, it is convenient to take $g_1=g_2=1$. This is the condition we will use in this paper to find strongly real Beauville structures.

Abelian strongly real Beauville groups are easy to determine. Catanese \cite{cat} proved that a finite abelian group is a Beauville group if and only if it is isomorphic to $C_n \times C_n$, where $n>1$ and $\gcd(n,6)=1$.  Since for any abelian group inversion is an automorphism,
every abelian Beauville group is a strongly real Beauville group.

Thus, if $p\geq 5$ there are  infinitely many abelian strongly real Beauville $p$-groups. If the $p$-group is non-abelian, it is harder to construct a strongly real Beauville structure.

Recall that in \cite{SV}, Stix and Vdovina constructed
infinite series of Beauville $p$-groups by considering quotients of ordinary triangle groups (von Dyck groups). In particular this gives examples of non-abelian Beauville $p$-groups of arbitrarily large order. On the other hand, the first explicit infinite family of Beauville $2$-groups was constructed in \cite[Theorem 1]{BBPV}. However, Theorem 1 in \cite{BBPV} also shows that with one exception these Beauville $2$-groups are not strongly real. 

The earliest examples of non-abelian strongly real Beauville $p$-groups were given by Fairbairn in \cite{fai}, by constructing the following pair of $2$-groups.
The groups
\[
G= \langle x,y \mid x^8=y^8=[x^2,y^2]=(x^iy^j)^4=1 \ \text{for} \ i,j=1,2,3 \rangle,
\]
and
\[
G= \langle x,y \mid (x^iy^j)^4=1 \ \text{for} \ i,j=0,1,2,3 \rangle
\]
are strongly real Beauville groups of order $2^{13}$ and $2^{14}$, respectively. In both cases, the Beauville structure is $\{ \{x,y\}, \{ xyx, xyxyx\} \}$.

Also in \cite{fai}, he asked the following question: ``Are there infinitely many strongly real Beauville $p$-groups?''

If $p\geq3$ the author  \cite{gul}  has recently given a positive answer to this question. She constructed an infinite family of non-abelian strongly real Beauville $p$-groups by considering the lower central quotients of the free product of two cyclic groups of order $p$. However, this result does not cover the prime $2$ and also does not give a strongly real Beauville $p$-group of every possible order.
At around the same time, Fairbairn \cite{fai2} gave another infinite family of non-abelian strongly real Beauville $p$-groups for odd $p$, by using wreath products of cyclic $p$-groups.

In this paper, we give a new infinite family of non-abelian strongly real Beauville $p$-groups for every prime $p$. As a consequence, we show that there are non-abelian strongly real Beauville $p$-groups of order $p^n$ for every $n\geq 3$ if $p\geq5$, or $n\geq 5$ if $p=3$, or $n\geq 7$ if $p=2$,  and hence there are strongly real Beauville  $p$-groups exactly for the same orders for which there exist Beauville $p$-groups.

In order to obtain the result, we work with quotients of the triangle group $T=\langle  a,b \mid a^q=b^q=(ab)^r=1 \rangle$ where $p$ is a fixed prime, $q=p^k>2$ and $r=p^{k+1}$ or $p^k$, according as $p=3$ or $p \neq 3$. 
For every quotient which is Beauville, we get many different  strongly real Beauville structures. Furthermore, we not only get infinitely many strongly real Beauville $p$-groups but also infinitely many different signatures, because the signature of one of the triples of Beauville structures takes the value $(p^k, p^k, p^k)$ if $p\neq 3$ or the value $(3^k, 3^k, 3^{k+1})$. All these results follow from the main theorem of this paper, which is as follows.

\begin{thmA}
	Let $p$ be a prime and let $T=\langle  a,b \mid a^q=b^q=(ab)^r=1 \rangle$ be the triangle group where $q=p^k>2$ for some $k \in \N$ and  $r= p^{k+1}$ if $p=3$ or $r=p^k$ if $p\neq3$. Then the following hold:
	\begin{enumerate}
		\item The lower central quotient $T/\gamma_n(T)$ is a strongly real Beauville group for $n\geq 3$ if $p>3$, or $n\geq 4$ if $p=2,3$.
		\item The series $\{\gamma_n(T)\}_{n\geq 3}$ if $p>3$ and $\{\gamma_n(T)\}_{n\geq 4}$ if $p=2$ or $3$ can be refined to a normal series of $T$ such that two consecutive terms of the series have index $p$ and for every term $N$ of the series $T/N$ is a strongly real Beauville group.
	\end{enumerate} 
\end{thmA}

The reader should be aware that when the prime $p=3$, we work on the quotients of the triangle group $\langle a,b \mid a^{3^k}=b^{3^k}=(ab)^{3^{k+1}}=1 \rangle$ rather than  $\langle a,b \mid a^{3^k}=b^{3^k}=(ab)^{3^k}=1 \rangle$. The reason why we need this change is explained at the end of the paper.

\vspace{10pt}

\noindent
\textit{Notation.\/}
We use standard notation in group theory. If $G$ is a group, then we denote by $\Cl_G(x)$ the conjugacy class of the element $x\in G$. Also, if $p$ is a prime, then we write $G^{p^i}$ for the subgroup generated by all powers $g^{p^i}$ as $g$ runs over $G$. The exponent of a $p$-group $G$, denoted by $\exp G$, is the maximum of the orders of all elements of $G$.

\section{ Proof of the main theorem}

Let $m,n,r \in \N$ and let $T$ be the triangle group defined by the presentation
\[
T=\langle a,b \mid a^{p^m}=b^{p^n}=(ab)^{p^r}=1\rangle.
\]
Stix and Vdovina \cite[Theorem 2]{ SV} showed that if there is a Beauville $p$-group $G$ where the signature of one of the triples of the Beauville structure is $(p^m,p^n,p^r)$,  then this Beauville structure of $G$ is inherited by infinitely many quotients of the triangle group $T$.

Note that for any triangle group $T$, the map $\theta \colon T\longrightarrow T$ defined by $\theta(a)=a^{-1}$ and $\theta(b)=b^{-1}$ is an automorphism. Thus, quotients of the triangle group seem to be good candidates for strongly real Beauville $p$-groups.

In this section, we give the proof of Theorem A.  Let $T$ be the triangle group as in Theorem A.
In order to determine Beauville structures in quotients of $T$, our starting point will be to analyze the quotient group $T/\gamma_3(T)$ if $p>3$, or $T/\gamma_4(T)$ if $p=2$ or $3$. To this purpose, we need to know the presentation of these quotient groups, and we have the following theorem.

\begin{thm}
	\label{special quotients}
	Let $p$ be a prime and let $T=\langle  a,b \mid a^q=b^q=(ab)^r=1 \rangle$ be the triangle group where $q=p^k>2$ for some $k \in \N$ and  $r= p^{k+1}$ if $p=3$ or $r=p^k$ if $p\neq3$. Then the following hold:
	\begin{enumerate}
		\item
		If $p>3$ then $T/\gamma_3(T)\cong G$ where
	   \[
		G=
		\langle x,y,z \mid x^{p^k}=y^{p^k}=z^{p^k}=1, [y,x]=z\rangle,
		\]
		and $\exp G=p^k$.
		\item
		If $p=3$ then $T/\gamma_4(T)\cong G$ where
		\begin{multline*}
		G=
		\langle x,y,z,t,w \mid x^{3^k}=y^{3^k}=z^{3^k}=t^{3^k}=w^{3^k}=1,
		\\
		[y,x]=z, [z,x]=t, [z,y]=w \rangle,
		\end{multline*}
		and $\exp G=3^{k+1}$.
		\item
		If $p=2$ then $T/\gamma_4(T)\cong G$ where
		\begin{multline*}
		G=
		\langle x,y,z,t,w \mid x^{2^k}=y^{2^k}=z^{2^{k-1}}=t^{2^{k-1}}=w^{2^{k-1}}=1,
		\\
		[y,x]=z, [z,x]=t, [z,y]=w \rangle,
		\end{multline*}
		and $\exp G=2^k$.
	\end{enumerate} 
	\end{thm}

\begin{proof}
	We first prove (i). Let $p>3$. Then the group $G$ given in (i) is the semidirect product of $\langle y\rangle \times \langle z\rangle\cong C_{p^k}\times C_{p^k}$ by $\langle x\rangle \cong C_{p^k}$. Since $G$ is of class $2$, for any $n\in \N$ we have $(xy)^{p^n}=x^{p^n}y^{p^n}[y,x]^{\binom{p^n}{2}}$. Hence $o(xy)=p^k$. 
	
	Let $u$ and $v$ be the images of $a$ and $b$ in $T/\gamma_3(T)$, respectively. Since $o(x)=o(y)=o(xy)=p^k$ and  $\gamma_3(G)=1$, the map
	$\alpha \colon T/\gamma_3(T) \longrightarrow G$ sending $u$ to $x$ and $v$ to $y$ is well-defined and an epimorphism. Thus $|T/\gamma_3(T)|\geq |G|=p^{3k}$. On the other hand, since $\gamma_2(T)/\gamma_3(T)$ is cyclic of exponent at most $p^k$, it then follows that $|T/\gamma_3(T)|\leq p^{3k}$, and hence $G\cong T/\gamma_3(T)$.
	
	We now prove (ii) and (iii) simultaneously. Let us  call $A= \langle z\rangle \times \langle t\rangle \times \langle w\rangle\cong C_s\times C_s\times C_s$, where $s=3^k$ for $k\geq 1$ or $2^{k-1}$ for $k\geq 2$. Then the group $G$ can be constructed as the semidirect product of $ \langle y \rangle \ltimes A$ by $\langle x \rangle\cong C_{p^k}$, where $\langle y \rangle\cong C_{p^k}$ and $p=2$ or $3$. Then $G$ is of class $3$. We next show that $o(xy)=p^{k}$ or $p^{k+1}$ according as $p=2$ or $3$.  First of all, one can show that for any $i\in \N$, we have that $[y,x^i]=z^it^{\binom{i}{2}}$. If we call $A=\sum_{i=1}^{n-1} \binom{i}{2}$ and $B=\sum_{i=1}^{n-1} i^2$, then for any $n\in \N$ we have
	\begin{equation*}
	\begin{split} 
	(xy)^n & =x^ny^{x^{n-1}}y^{x^{n-2}}\dots y^xy 
	= x^nyz^{n-1}yz^{n-2}\dots zy t^{A}\\
	&=x^ny^n(z^{n-1})^{y^{n-1}}\dots z^y t^{A}
	=x^ny^nz^{\binom{n}{2}} t^{A}w^{B}.
	\end{split}
	\end{equation*}
	Note that $A=\binom{n}{3}$ and $B=\frac{(n-1)n(2n-1)}{6}$.
	It then follows that $o(xy)=2^k$ if $p=2$, or $o(xy)=3^{k+1}$ if $p=3$.
	
	As before, we have an epimorphism from $T/\gamma_4(T)$ to $G$. Thus $|T/\gamma_4(T)|\geq |G|=2^{5k-3}$ or $3^{5k}$. On the other hand, if $p=3$ then  since $\gamma_2(T)/\gamma_3(T)$ is cyclic and $\gamma_3(T)/\gamma_4(T)$ is a $2$-generator group, and both are of exponent at most $3^k$,  it then follows that $|T/\gamma_4(T)|\leq 3^{5k}$. Hence $G\cong T/\gamma_4(T)$. 
	If $p=2$ then
	\begin{equation*}
	\begin{split}
	1 &=(ab)^{2^k}\equiv  a^{2^k}b^{2^k}[b,a]^{\binom{2^k}{2}}
	\equiv [b,a]^{\binom{2^k}{2}} \pmod{\gamma_3(T)}.
	\end{split}
	\end{equation*}
	Thus $[b,a]$ is of order at most $2^{k-1}$ modulo $\gamma_3(T)$. It then follows that $ \gamma_2(T)/\gamma_3(T)$ and $\gamma_3(T)/\gamma_4(T)$ are of exponent $\leq 2^{k-1}$. This implies that  $|T/\gamma_4(T)|\leq 2^{5k-3}$, and hence $T/\gamma_4(T)\cong G$, as desired.
	
	Now we show that $\exp G=p^k$ if $p\neq 3$. If $p>3$ then $G$ is regular and hence $\exp G =p^k$. If $p=2$ and $g,h\in G$, then by the Hall-Petrescu formula (see \cite[III.9.4]{hup}), we have
	\[
	(gh)^{2^k}=g^{2^k}h^{2^k}c_2^{\binom{2^k}{2}}c_3^{\binom{2^k}{3}},
	\]
	where $c_i \in \gamma_i(\langle g,h \rangle)$. Since $\exp G'=2^{k-1}$, it follows that $(gh)^{2^k}=g^{2^k}h^{2^k}$. Since $G$ is generated by two elements of order $2^k$, we conclude that $\exp G=2^k$.
	A similar calculation shows that if $p=3$ then $\exp G=3^{k+1}$.
\end{proof}

In order to prove Theorem A, our first aim will be to show that the quotients given in Theorem \ref{special quotients} are strongly real Beauville groups. We deal separately with the cases $p>3$, $p=3$ and $p=2$.

We start with the case $p>3$. Note that in this case, the quotient group $T/\gamma_3(T)$ is of class $2<p$, and thus it is a regular $p$-group. 
The following result gives a necessary and sufficient condition for a regular $p$-group to be a Beauville group.

\begin{lem}\textup{\cite[Corollary~2.6]{FG} }
	\label{regular p-group}
	Let $G$ be a finite $2$-generator regular $p$-group. Then $G$ is a Beauville group if and only if $p\geq5$ and $|G^{p^{e-1}}|\geq p^2$, where $\exp G=p^e$. If that is the case, then every lift to $G$ of a Beauville structure of $G/\Phi(G)$ is a Beauville
	structure of $G$.
\end{lem}

\begin{lem}
	\label{p>3 first quotient}
	Let $p>3$, $k\geq 1$ and let
	\[
	G=
	\langle x,y,z \mid x^{p^k}=y^{p^k}=z^{p^k}=1, [y,x]=z\rangle.
	\]
Then $G$ is a strongly real Beauville group. More precisely, if $w_1$ and $w_2$ are any two symmetric words in $x$ and $y$ such that no two of the elements in the set $\{ x,y,xy,w_1,w_2,w_1w_2\}$ lie in the same maximal subgroup, then $\{x,y\}$ and $\{w_1, w_2\}$ form a strongly real Beauville structure for $G$.
\end{lem}

\begin{proof}
	First of all, observe that since $G$ is regular, we have $\exp G=p^k$ (see \cite[Theorem 3.14]{suz}). Also $|G^{p^{k-1}}|\geq p^3$. It then follows from Lemma \ref{regular p-group} that $G$ is a Beauville group.
	
	 Since $p\geq 5$,  $G/\Phi(G)$ is a Beauville group with the Beauville structure  $\{x\Phi(G), y\Phi(G)\}$ and $\{w_1\Phi(G),$ $w_2\Phi(G)\}$. According to Lemma \ref{regular p-group}, this Beauville structure is inherited by $G$, and hence $\{x,y\}$ and $\{w_1,w_2\}$ form a Beauville structure for $G$. 
	
	We next show that this Beauville structure is strongly real. The map $\theta \colon T\longrightarrow T$ defined by $\theta(a)=a^{-1}$ and $\theta(b)=b^{-1}$ is an automorphism. Since $G\cong T/\gamma_3(T)$, $\theta$ induces an automorphism $\phi \colon G \longrightarrow G$ defined by $\phi(x)=x^{-1}$ and $\phi(y)=y^{-1}$. Also note that since $w_1$ and $w_2$ are symmetric words in $x$ and $y$, we have $\phi(w_1)=w_1^{-1}$ and $\phi(w_2)=w_2^{-1}$. Hence $G$ is a strongly real Beauville group.
\end{proof}

Notice that it is always possible to choose two symmetric words $w_1$ and $w_2$ such that each element in the set 
$\{ x,y,xy,w_1,w_2,w_1w_2\}$ falls into a different maximal subgroup. For example, we can take $w_1=(xy)^{n_1}x$ and $w_2=(xy)^{n_2}x$ where $n_1 \equiv 1 \pmod{p}$ and $n_2\equiv 3 \pmod{p}$.

\vspace{5pt}
We next deal with the cases $p=2$ or $3$. To this purpose, we also need the following easy lemma.

\begin{lem}
	\label{orders preserve}
	Let $G=\langle a,b \rangle$ be a $2$-generator $p$-group and suppose that $G/G'=\langle aG'\rangle \times \langle bG'\rangle$. If $o(a)=o(aG')$ then
	\[
	\Big(\bigcup_{g\in G} {\langle a\rangle}^g  \Big)
	\bigcap
	\Big(\bigcup_{g\in G} {\langle b\rangle}^g \Big)= 1.
	\]
\end{lem}

\begin{proof}
	Let $x=(a^i)^g=(b^j)^h$ for some $g,h\in G$ and $i,j \in \Z$. Then in the quotient group $\overline{G}= G/G'=\langle \overline{a}\rangle\times \langle \overline{b}\rangle$, we have $\overline{x}= \overline{1}$, and thus $x\in G'\cap \langle a^g\rangle$. Since $o(a)=o(aG')$, we have $G'\cap \langle a^g\rangle=1$, and hence $x=1$, as desired.
\end{proof}

In the calculations in Lemmas \ref{p=2} and $\ref{p=3}$, we use the following identities.
Let $G$ be a group and let $x,y \in G$.
\begin{enumerate}
	\item 
	If $G'$ is abelian, then
	\[
	[x,y^n]= [x,y]^n[x,y,y]^{\binom{n}{2}} \dots [x,y,\overset{n}{\ldots},y]^{\binom{n}{n}}.
	\]
	\item
	If $\langle y, G'\rangle$ is abelian, then
	\[
	(xy)^n=x^ny^n[y,x]^{\binom{n}{2}}\dots [y,x,\overset{n-1}{\ldots},x]^{\binom{n}{n}}.
	\]
\end{enumerate}

\begin{lem}
	\label{p=2}
	Let $k>1$ and let \begin{multline*}
	G=
	\langle x,y,z,t,w \mid x^{2^k}=y^{2^k}=z^{2^{k-1}}=t^{2^{k-1}}=w^{2^{k-1}}=1, [y,x]=z,
	\\
	[z,x]=t, [z,y]=w \rangle.
	\end{multline*}
	Then $G$ is a strongly real Beauville group. More precisely, if $w_1=(xy)^{n_1}x$ and $w_2=(xy)^{n_2}x$ for $n_1, n_2 \in \Z^+$ such that $n_1 \equiv 1 \pmod{4}$ and $n_2\equiv 2 \pmod{4}$, then $\{x,y\}$ and $\{w_1, w_2\}$ form a strongly real Beauville structure for $G$.
\end{lem}

\begin{proof}
	Let $A= \{x,y, xy\}$ and $B=\{w_1, w_2, w_1w_2\}$. We need to show that
	\begin{equation}
	\label{check}
	\langle a^g\rangle
	\cap
	\langle b\rangle=1,
	\end{equation}
	for all $a\in A$, $b\in B$, and $g\in G$.
	
	Recall that by the proof of Theorem \ref{special quotients}, $\exp G=2^k$ and also $o(xy)=2^k$, and hence $o(a)=2^k$ for every $a \in A$. Thus if (\ref{check}) does not hold, then
	$b^{2^{k-1}} \in \langle (a^{2^{k-1}})^g\rangle$ for some $g\in G$.
	
	Let us start with the case $a=x$ and $b=w_2$. Set $\overline{G}=G/ \langle t\rangle$. Then it can be seen that $\overline{x}^{2^{k-1}}$ is central in $\overline{G}$, that is,
	$
	\Cl_{G}(x^{2^{k-1}}) \subseteq  x^{2^{k-1}}\langle t \rangle.
	$
	On the other hand, observe that the following general formula holds, which follows from the identities given above. If $g,h \in G$, $c\in G'$ and $ \gamma_2(\langle h, G'\rangle) \leq \langle t \rangle$, then
	\begin{equation}
	\label{general eqn}
	\begin{split}
	(g^{2}hc)^{2^{k-1}}& \equiv g^{2^k}h^{2^{k-1}}
	[h,g,g]^{\binom{2^{k-1}}{2}} \pmod{\langle t \rangle}.
	\end{split}
	\end{equation}
	 Since $w_2=y^{n_2}x^{n_2+1}c$, where $c \in G'$ and $n_2=4m+2$ for some $m\in \N$, it follows from formula (\ref{general eqn}) that
		\begin{align*}
		w_2^{2^{k-1}}
		\equiv
		x^{(n_2+1)2^{k-1}} w^{N} \pmod{\langle t \rangle},
		\end{align*}
    where
    \[
    N= (n_2+1)(2m+1)^2\binom{2^{k-1}}{2} \vspace{5pt}
    \]
    is not divisible by $2^{k-1}$.
	Hence $w_2^{2^{k-1}}\notin \Cl_G(x^{2^{k-1}})$.
	
	Now assume that $a=y$ and $b=w_1$. Set $\overline{G}=G/\langle w\rangle$. Then $\overline{y}^{2^{k-1}}$ is central in $\overline{G}$, that is,
	$\Cl_{G}(y^{2^{k-1}}) \subseteq
	 y^{2^{k-1}}\langle w \rangle$.
	On the other hand, we have $w_1= x^{n_1+1}y^{n_1}c$ for some $c\in G'$, and $n_1+1=4m+2$ for some $m \in \N$. Then applying formula (\ref{general eqn}) modulo $\langle w \rangle$, we get
	\begin{align*}
	w_1^{2^{k-1}}
	\equiv
	y^{n_1{2^{k-1}}}t^{N} \pmod{ \langle w \rangle},
	\end{align*}
	for some $N$ which is not divisible by $2^{k-1}$, and hence
	$w_1^{2^{k-1}}\notin \Cl_G(y^{2^{k-1}})$.
	
	 We now consider the case $a=xy$ and $b=w_1w_2$. Set $\overline{G}=G/\langle tw\rangle$. Then  $\overline{xy}^{2^{k-1}}$ is central, that is,
	 $
	 \Cl_{G}((xy)^{2^{k-1}}) \subseteq
	 (xy)^{2^{k-1}}\langle tw \rangle .
	 $
	 Note that $w_1w_2=x^2(xy)^{n_1+n_2}c$ for some $c\in G'$. As before,
	 \begin{align*}
	  (w_1w_2)^{2^{k-1}} \equiv (xy)^{2^{k-1}}t^{(n_1+n_2)\binom{2^{k-1}}{2}} \pmod{\langle tw\rangle},
	  \end{align*}
	  where $2 \nmid n_1+n_2$, and hence $(w_1w_2)^{2^{k-1}}\notin \Cl_G((xy)^{2^{k-1}})$.

	Observe that in the remaining cases $a$ and $b$ lie in two different maximal subgroups of $G$, and so  $G=\langle a,b \rangle$. Since $G/G'\cong C_{2^k}\times C_{2^k}$ and $o(aG')=o(bG')=2^k$, it then follows that 
    $G/G'=\langle aG' \rangle\times \langle bG'\rangle$. Therefore in all these cases we can apply  Lemma \ref{orders preserve}. 
	This completes the proof that $G$ is a Beauville group.
	
	We know that $G\cong T/\gamma_4(T)$, where $T$ is the triangle group given in Theorem A. Since the automorphism $\theta$ of $T$ induces an automorphism of $T/\gamma_4(T)$ and $w_1$, $w_2$ are symmetric words in $x$ and $y$, we conclude that the Beauville structure $\{\{x,y\}, \{w_1, w_2\} \}$ is strongly real.
\end{proof}

\begin{lem}
	\label{p=3}
	Let $k\geq1$ and let \begin{multline*}
	G=
	\langle x,y,z,t,w \mid x^{3^k}=y^{3^k}=z^{3^k}=t^{3^k}=w^{3^k}=1, [y,x]=z,
	\\
	[z,x]=t, [z,y]=w \rangle.
	\end{multline*}
	Then $G$ is a strongly real Beauville group. More precisely, if $w_1=(xy)^{n_1}x$ and $w_2=(xy)^{n_2}x$ for $n_1, n_2 \in \Z^+$ such that $n_1 \equiv 1 \pmod{9}$ and $n_2\equiv 2 \pmod{9}$, then $\{x,y\}$ and $\{w_1, w_2\}$ form a strongly real Beauville structure for $G$.
\end{lem}

\begin{proof}
	First of all, we will show that for any $g\in G$ and $h\in \Phi(G)$, we have
	\begin{equation}
	\label{power relation}
	(gh)^{3^k}=g^{3^k}.
	\end{equation}
	By the Hall-Petrescu formula, we have
	\begin{equation*}
	\label{hall-petrescu}
	(gh)^{3^{k}}= g^{3^{k}}h^{3^{k}}c_2^{\binom{3^{k}}{2}}c_3^{\binom{3^{k}}{3}} ,
	\end{equation*}
	where $c_i \in \gamma_i(\langle g,h \rangle)$. Note that $\exp G'=3^k$. Also notice that $\langle g,h\rangle' \leq \langle g, \Phi(G)\rangle' \leq [G,\Phi(G)]$. Then  
	\[
    \gamma_3(\langle g,h\rangle)\leq	[\Phi(G),G,G]=[G'G^3,G,G]=\gamma_3(G)^3.
	\]
	Thus $(gh)^{3^k}=g^{3^k}h^{3^k}$.
	On the other hand, since $\Phi(G)=\langle x^3,y^3, G'\rangle$ and since $\exp G'=3^k$ and $o(x^3)=o(y^3)=3^{k-1}$, it follows that $\exp \Phi(G)=3^k$, and hence $(gh)^{3^k}=g^{3^k}$, as desired.
	
	Let $A= \{x,y, xy\}$ and $B=\{w_1, w_2, w_1w_2\}$.
    Let us start with the case $a=xy$ and $b=w_1$. Observe that $b$ and $xy^2$ lie in the same maximal subgroup of $G$, and this, together with  (\ref{power relation}), implies that $\langle (b^g)^{3^k} \rangle=\langle (xy^2)^{3^k} \rangle$, for all $g\in G$. Recall that by the proof of Theorem \ref{special quotients}, for any $n \in \N$ we have
    \begin{equation}
    \label{powers of xy}
    (xy)^{n}=x^ny^nz^{\binom{n}{2}}t^{M}w^{N},
    \end{equation}
    where  $N=\frac{(n-1)n(2n-1)}{6}$ and $M=\binom{n}{3}$. A similar calculation shows that
	\begin{equation*}
	\begin{split} 
	(xy^2)^n &
	=x^ny^{2n}z^{2\binom{n}{2}} t^{2M}w^{4N+\binom{n}{2}}.
	\end{split}
	\end{equation*}
	Thus both $xy$ and $xy^2$ are of order $3^{k+1}$. If we take $n=3^k$, then  $(xy)^{3^k}=t^{\binom{3^k}{3}}w^{N}$ and  $(xy^2)^{3^k}=t^{2\binom{3^k}{3}}w^{4N}$, where $N= \frac{3^k(3^k-1)(23^k-1)}{6}$. Consequently, $\langle (xy)^{3^k} \rangle\cap \langle (b^g)^{3^k}\rangle=1$. 
	
	Next we assume that $a=y$ and $b=w_2$. Notice that $w_2 \equiv y^2 \pmod{\Phi(G)}$. Set $\overline{G}=G/\langle w\rangle$. Then we have
	\[
	\Cl_{\overline{G}}((\overline{y}^2)^{3^{k-1}})=
	\{ (\overline{y}^2)^{3^{k-1}}(\overline{z}^{2i}
	\overline{t}^{2\binom{i}{2}})^{3^{k-1}} \mid i=0,1,2 \}.
	\]
	On the other hand, as in the proof of Lemma \ref{p=2},
	it can be similarly shown that if $g,h \in G$, $c\in G'$ and $ \gamma_2(\langle h, G'\rangle) \leq \langle w \rangle$, then
	\begin{equation}
	\label{general eqn for p=3}
	\begin{split}
	(g^{3}hc)^{3^{k-1}}& \equiv g^{3^k}h^{3^{k-1}}c^{3^{k-1}} \pmod{\langle w \rangle}.
	\end{split}
	\end{equation}
    Observe that
    \[
    (xy)^{n_2}x \equiv x^{n_2+1}y^{n_2}z^{\binom{n_2+1}{2}} t^{\binom{n_2+1}{3}} \pmod{\langle w \rangle}.
    \]
    Thus by  formula (\ref{general eqn for p=3}), we get
	\begin{align}
	\label{power of w_2}
	w_2^{3^{k-1}}
	\equiv
	(y^2)^{3^{k-1}}(z^{\binom{n_2+1}{2}}t^{\binom{n_2+1}{3}})^{3^{k-1}}
	\equiv
	(y^2)^{3^{k-1}}t^{3^{k-1}}
	 \pmod{\langle w \rangle},
	\end{align}
	since $n_2\equiv 2 \pmod{9}$.
	Therefore, we have
	$\langle (a^g)^{3^{k-1}}\rangle \neq \langle b^{3^{k-1}} \rangle$ for any $g\in G$. Since by formulas (\ref{power relation}) and (\ref{power of w_2}), $o(w_2)=3^k$, we conclude that $\langle a^g \rangle \cap \langle b\rangle=1$ for any $g\in G$.

	Now we consider the case $a=x$ and $b=w_1w_2$. Note that $w_1w_2 \equiv x^2 \pmod{\Phi(G)}$. Set $\overline{G}=G/\langle t\rangle$. Then we have
	\[
	\Cl_{\overline{G}}((\overline{x}^2)^{3^{k-1}})=
	\{ (\overline{x}^2)^{3^{k-1}}(\overline{z}^{\ -2i}
	\overline{w}^{\ -2\binom{i}{2}})^{3^{k-1}} \mid i=0,1,2 \}.
	\]
 On the other hand, observe that $w_1w_2 \equiv (xy)^{n_1+n_2}x^2z^{-n_2}w^{-\binom{n_2}{2}} \pmod{\langle t \rangle}$.
 Then by applying formula (\ref{general eqn for p=3}) modulo $\langle t\rangle$ and by taking into account formula (\ref{powers of xy}), we get
\begin{equation}
\label{power of w_1w_2}
 \begin{split} 
 (w_1w_2)^{3^{k-1}}& \equiv (xy)^{(n_1+n_2)3^{k-1}}(x^2)^{3^{k-1}}z^{-n_23^{k-1}}w^{-\binom{n_2}{2}3^{k-1}}\\
 & \equiv (x^2)^{3^{k-1}}z^{-n_23^{k-1}} w^{-\binom{n_2}{2}3^{k-1}+N} \pmod{\langle t \rangle},
 \end{split}
 \end{equation}
 where $N= \sum_{i=1}^{s3^k-1}i^2 \equiv- s3^{k-1} \pmod{3^k}$ and $n_1+n_2=3s$ for some $s \equiv 1 \pmod{3}$.
 Observe that there is no $i\in \N$ such that $2i \equiv n_2 \pmod{3}$ and $\binom{n_2}{2}+s \equiv 2\binom{i}{2} \pmod{3}$. Thus $\langle (a^g)^{3^{k-1}}\rangle \neq \langle b^{3^{k-1}} \rangle$ for any $g\in G$. Since $o(w_1w_2)=3^k$, by (\ref{power relation}) and (\ref{power of w_1w_2}), we conclude that $\langle a^g\rangle \cap \langle b \rangle=1$ for all $g\in G$, as desired.  
 
 We next deal with the cases $a=x$ and $b=w_1$ or $w_2$, or $a=y$ and $b=w_1$ or $w_1w_2$, or $a=xy$ and $b=w_2$ or $w_1w_2$. In all these cases, we have $G=\langle a,b \rangle$. Since $o(aG')=o(bG')=3^k$ and $G/G'\cong C_{3^k}\times C_{3^k}$, it follows that $G/G'=\langle aG' \rangle\times \langle bG'\rangle$. Also notice that one of the two elements $a$ or $b$ has  the same order, namely $3^k$, in both $G$ and $G/G'$. Hence we apply Lemma \ref{orders preserve}.
\end{proof}

The following result, which gives a sufficient condition to lift a Beauville structure from a quotient group, is Lemma 4.2 in \cite{FJ}.

\begin{lem}
	\label{lifting structure}
	Let $G$ be a finite group and let $\{x_1,y_1\}$ and $\{x_2,y_2\}$ be two sets of generators of $G$.
	Assume that, for a given $N\trianglelefteq G$, the following hold:
	\begin{enumerate}
		\item
		$\{x_1N,y_1N\}$ and $\{x_2N,y_2N\}$ form a Beauville structure for $G/N$.
		\item
		$o(g)=o(gN)$ for every $g\in\{x_1,y_1,x_1y_1\}$.
	\end{enumerate}
	Then $\{x_1,y_1\}$ and $\{x_2,y_2\}$ form a Beauville structure for $G$.
\end{lem}

We are now ready to give the proof of Theorem A.

\begin{thm}
	\label{main theorem}
	Let $p$ be a prime and let $T=\langle  a,b \mid a^q=b^q=(ab)^r=1 \rangle$ be the triangle group where $q=p^k>2$ for some $k \in \N$ and  $r= p^{k+1}$ if $p=3$ or $r=p^k$ if $p\neq3$. Then the following hold:
	\begin{enumerate}
		\item The lower central quotient $T/\gamma_n(T)$ is a strongly real Beauville group for $n\geq 3$ if $p>3$, or $n\geq 4$ if $p=2,3$.
		\item The series $\{\gamma_n(T)\}_{n\geq 3}$ if $p>3$ and $\{\gamma_n(T)\}_{n\geq 4}$ if $p=2$ or $3$ can be refined to a normal series of $T$ such that two consecutive terms of the series have index $p$ and for every term $N$ of the series $T/N$ is a strongly real Beauville group.
	\end{enumerate} 
\end{thm}

\begin{proof}
	Let $\theta$ be the automorphism of $T$ defined by $\theta(a)=a^{-1}$ and $\theta(b)=b^{-1}$.
	If we call $S$ the set of all commutators of length $i$ in $a$ and $b$, then  $\gamma_i(T)/\gamma_{i+1}(T)$ is generated by elements in $S$ modulo $\gamma_{i+1}(T)$. Also note that
	 \begin{align*}
	\theta([x_{j_1}, x_{j_2}, \dots, x_{j_i}])&=[x_{j_1}^{-1}, x_{j_2}^{-1}, \dots, x_{j_i}^{-1}]\\	
	&\equiv [x_{j_1}, x_{j_2}, \dots, x_{j_i}]^{\delta} \pmod{\gamma_{i+1}(T)},
	\end{align*}
	where each $x_{j_k}$ is either $a$ or $b$ and $\delta \in \{-1,1\}$. 
	Thus, by adding the elements in $S$ one by one to $\gamma_{i+1}(T)$, we produce a series of normal subgroups between $\gamma_{i+1}(T)$ and $\gamma_i(T)$ such that each normal subgroup is invariant under $\theta$.
	
	Let $M=\langle s_1, s_2, \dots, s_{\ell-1}, \gamma_{i+1}(T)\rangle$ and $L=\langle s_1, s_2, \dots, s_{\ell}, \gamma_{i+1}(T)\rangle$ be two consecutive terms of the above series, where each $s_k \in S$ and  $|L:M|=p^m$ for some $m\in \N$. If we set  $K_i=\langle M, s_{\ell}^{p^{i}}\rangle$ for $0\leq i \leq m$, then we get a chain of $\theta$-invariant normal subgroups
	\[
	M=K_m\leq K_{m-1}\leq \dots \leq K_1 \leq K_0=L 
	\] 
	such that two consecutive terms have index $p$.

	Hence we can refine the series $\{\gamma_n(T)\}_{n\geq 3}$ if $p>3$ and $\{\gamma_n(T)\}_{n\geq 4}$ if $p=2$ or $3$ to a normal series of $T$ such that two consecutive terms of the series have index $p$ and every term $N$ is invariant under $\theta$.
	
	We will see that $H=T/N$ is a strongly real Beauville group, which simultaneously proves (i) and (ii).  Let us call $u$ and $v$ the images of $a$ and $b$ in $T/\gamma_3(T)$, respectively.
		
	If $p>3$ then by Lemma \ref{p>3 first quotient}, we know that $T/\gamma_3(T) \cong G$ is a Beauville group with the Beauville structure $\{\{x,y\}, \{w_1, w_2\} \}$, where $x$ is sent to $u$ and $y$ is sent to $v$ by the isomorphism from $G$ to $T/\gamma_3(T)$. On the other hand, note that  $o(a)=o(b)=o(ab)=p^k$ modulo $\gamma_3(T)$ and modulo $N$. Then according to Lemma \ref{lifting structure}, the Beauville structure of $T/\gamma_3(T)$ is inherited by $H$.
    Similarly, if $p=2$ or $3$, the Beauville structure of $T/\gamma_4(T)$ given in Lemmas \ref{p=2} and \ref{p=3} is inherited by $H$.
	
	Since $N$ is invariant under $\theta$, the map $\theta$ induces an automorphism of $H$. Thus, clearly the Beauville structures are strongly real. This completes the proof.
\end{proof}

As a consequence of Lemmas \ref{p>3 first quotient}, \ref{p=2} and \ref{p=3}, the quotients of the triangle groups given in Theorem \ref{main theorem} have many different strongly real Beauville structures.

We close the paper by showing why we do not use the quotients of the triangle group $T=\langle a,b \mid a^{3^k}=b^{3^k}=(ab)^{3^k}=1\rangle$, giving a uniform treatment for all primes. In this case, unlike in Lemma \ref{p=3}, the quotient group $T/\gamma_4(T)$ is not a Beauville group.

\begin{lem}
Let $T=\langle  a,b \mid a^{3^k}=b^{3^k}=(ab)^{3^k}=1 \rangle$. Then $T/\gamma_4(T)$ is not a Beauville group.
\end{lem}

\begin{proof}
Let us call $G$ the quotient group 	$T/\gamma_4(T)$, and let $x$ and $y$ be the images of $a$ and $b$ in $G$, respectively. Since $G'$ is abelian, we have
\[
1=[x,y^{3^k}]=[x,y]^{3^k}[x,y,y]^{\binom{3^k}{2}}=[x,y]^{3^k}.
\]
The last equality follows from the fact that $\exp \gamma_3(G) | \exp  G/G'=3^k$. Thus $\exp G' | 3^k$.
Put $\overline{G}= G/\langle [x,y,x]\rangle$. Then $\langle \overline{x}, \overline{G'}\rangle$ is abelian, and thus
\[
\overline{1}=(\overline{yx})^{3^k}= \overline{y}^{3^k}\overline{x}^{3^k}[\overline{x}, \overline{y}]^{\binom{3^k}{2}}[\overline{x}, \overline{y}, \overline{y}]^{\binom{3^k}{3}}=[\overline{x}, \overline{y}, \overline{y}]^{3^{k-1}}.
\]
 Therefore, $[x,y,y]^{3^{k-1}} \in \langle [x,y,x]\rangle$. Since there is an automorphism of $G$ exchanging $x$ and $y$, we have $o([x,y,y])=o([x,y,x])$. Therefore, $[x,y,y]^{3^{k-1}} \in \langle [x,y,x]^{3^{k-1}}\rangle$.

Thus, if $k=1$ then $\gamma_3(G)= \langle [x,y,x] \rangle$, which is of order $3$, and hence $|G|=3^4$. Consequently, $G$ is not a Beauville group.

We now assume that $k>1$. For any $g \in G$ if we write $g=x_1x_2\dots x_n$, where $x_i=x$ or $y$, then by using induction on $n$ and by using the Hall-Petrescu formula, it can be easily seen that $g^{3^k} \in \gamma_3(G)^{3^{k-1}}$. Also for any $u\in \Phi(G)=G^3G'$, if we write $u=v^3c$ for some $v\in G$ and $c\in G'$, then we get $u^{3^{k-1}} \in (G')^{3^{k-1}}$.

We next show that for every $g\in x\Phi(G)$, we have \begin{equation}
\label{same powers}
g^{3^{k-1}} \in x^{3^{k-1}}(G')^{3^{k-1}}.
\end{equation} If we write $g=xu$ for some $u \in \Phi(G)=G^3G'$, then we have
 \[
 g^{3^{k-1}}=x^{3^{k-1}}u^{3^{k-1}}c_2^{\binom{3^{k-1}}{2}}c_3^{\binom{3^{k-1}}{3}},
 \]
 where $c_i \in \gamma_i(\langle x, u\rangle)$ for $i=1,2$. Observe that $\gamma_3(\langle x, u\rangle) \leq \gamma_3(G)^3$, and this together with $u^{3^{k-1}} \in (G')^{3^{k-1}}$, yields that $g^{3^{k-1}} \in x^{3^{k-1}}(G')^{3^{k-1}}$.
 
 Now we will show that
  \begin{equation}
  \label{cl of x}
  \Cl_G(x^{3^{k-1}})= \{ x^{3^{k-1}}[x^{3^{k-1}},g]  \mid g\in G \}=x^{3^{k-1}}(G')^{3^{k-1}}.
 \end{equation}
 For any $g\in G$, we have
 \[
 [g,x^{3^{k-1}}]=[g,x]^{3^{k-1}}[g,x,x]^{\binom{3^{k-1}}{2}} \in (G')^{3^{k-1}}.
 \]
 Thus $\Cl_G(x^{3^{k-1}}) \subseteq x^{3^{k-1}}(G')^{3^{k-1}}$. On the other hand, since $G'$ is abelian and $[x,y,y]^{3^{k-1}} \in \langle [x,y,x]^{3^{k-1}} \rangle$, we have $(G')^{3^{k-1}}=\langle  [y,x]^{3^{k-1}}, [x,y,x]^{3^{k-1}}\rangle$. Notice that
\[
[[y^{-1},x]y, x^{3^{k-1}}]=[y,x]^{3^{k-1}}, \ \text{and} \ 
[[x,y], x^{3^{k-1}}]=[x,y,x]^{3^{k-1}}.
\]
Then for any $i,j \in \N$, we have
\[
\big[\big([y^{-1},x]y \big)^i[x,y]^j,  x^{3^{k-1}}\big]= \big([y,x]^i[x,y,x]^j\big)^{3^{k-1}},
\]
and hence $(G')^{3^{k-1}} \subseteq \{[x^{3^{k-1}},g]  \mid g\in G\}$. Therefore, $x^{3^{k-1}}(G')^{3^{k-1}} \subseteq \Cl_G(x^{3^{k-1}}) $.

We finally show that $G$ is not a Beauville group. Suppose, on the contrary, that $\{x_1, y_1\}$ and $\{x_2, y_2\}$ form a Beauville structure for $G$. Set $A=\{x_1,y_1,x_1y_1\}$ and $B=\{x_2,y_2,x_2y_2\}$. Since $G$ has $4$ maximal subgroups, we have a collision
$\langle a \rangle\Phi(G)=\langle b\rangle \Phi(G)$ for some $a\in A$, $b\in B$, and actually it is the same as one of the maximal subgroups $\langle x \rangle \Phi(G)$, $\langle y \rangle \Phi(G)$ or $\langle xy \rangle \Phi(G)$. Now we will apply an automorphism of $G$ on the Beauville structure, and this gives another Beauville structure for $G$.
Observe that we can always find an automorphism of $T$ sending $b$ to $a$ or $ab$ to $a$. Thus, $G$ has an automorphism sending $y$ to $x$ or $xy$ to $x$. By applying this automorphism on the Beauville structure, we move the collision to the maximal subgroup $\langle x\rangle \Phi(G)$. However, equations (\ref{same powers}) and (\ref{cl of x}) yield that for any $g\in \langle x \rangle \Phi(G)\smallsetminus \Phi(G)$, $g^{3^{k-1}} \in \langle x^{3^{k-1}}\rangle^{h}$ for some $h\in G$. Consequently, $G$ cannot be a Beauville group.
\end{proof}

\section*{Acknowledgments}
I would like to thank G. Fern\'andez-Alcober for helpful comments and suggestions. I thank Ben Fairbairn for providing me with his preprints \cite{fai2} and \cite{fai3}. Also, I wish to thank the Department of Mathematics at the University of the Basque Country for its hospitality while this paper was being written. 

In order to test some groups of small order, we have used the computer algebra system GAP \cite{GAP}.

\end{document}